\documentclass{amse-newgai}

\numberwithin{equation}{section} 

\begin{document}

 \PageNum{1}
 \Volume{201x}{Sep.}{x}{x}
 \OnlineTime{xxx xx, 201x}
 \DOI{0000000000000000}

\abovedisplayskip 6pt plus 2pt minus 2pt \belowdisplayskip 6pt
plus 2pt minus 2pt
\def\vsp{\vspace{1mm}}
\def\th#1{\vspace{1mm}\noindent{\bf #1}\quad}
\def\proof{\vspace{1mm}\noindent{\it Proof}\quad}
\def\no{\nonumber}
\newenvironment{prof}[1][Proof]{\noindent\textit{#1}\quad }
{\hfill $\Box$\vspace{0.7mm}}
\def\q{\quad} \def\qq{\qquad}
\allowdisplaybreaks[4]


\AuthorMark{Cui Z., Li H. Z. and Xue B. Q.}                             

\TitleMark{On sums of sparse prime subsets}  

\title{On sums of sparse prime subsets     
\footnote{This work is supported by the National Natural Science Foundation of China (Grant No. 10771135, 11271249) and Innovation Program of Shanghai Municipal Education Commission.}}                  

\author{Zhen CUI \quad Hong Ze LI \quad Bo Qing XUE*\footnote{* The third author is the corresponding author.}}             
    {Department of Mathematics, Shanghai Jiao Tong University, Shanghai 200240, P. R. China\\
    E-mail\,$:$ zcui@sjtu.edu.cn \quad lihz@sjtu.edu.cn \quad ericxue68@sjtu.edu.cn}

\maketitle%

\Abstract{For arbitrary $c_0>0$, if $A$ is a subset of the primes less than $x$ with cardinality $\delta x (\log x)^{-1}$ with $\delta\geq (\log x)^{-c_0}$, then there exists a positive constant $c$ such that the cardinality of $A+A$ is larger than $c\, \delta x (\log\log x)^{-1}$.}      


\MRSubClass{11P32, 11P55, 11B05, 11B13}      

\section{Introduction}

The Goldbach conjecture is one of the oldest unsolved problems in number theory. The ternary case was basically solved by Vinogradov, showing that every sufficiently large odd integer can be expressed as the sum of three primes. For the binary Goldbach problem, people can only get `almost all' results. For example, Lu \cite{Lu} proved that the number of even integers $n$ not exceeding $x$ for which $n$ is not the sum of two primes is $\textit{O}(x^{0.879})$.

On the other hand, additive properties of the primes have been widely studied in recent years. Van der Corput \cite{Cor} showed that the primes contain infinitely many three-term arithmetic progressions. Green \cite{G} proved that three-term arithmetic progression exists in subsets of the primes with positive relative density. And a celebrated work by Green and Tao \cite{G-T} proved that the primes contain arbitrarily long arithmetic progressions.

Combining these two kinds of problems, one may wish to find properties of $A+A$, with $A$ a subset of the primes. Using `W-trick', the strategy developed by Green-Tao, Chipeniuk and Hamel \cite{C-H} showed that if $A$ is a subset of the primes with positive relative lower density $\delta$, then the set $A+A$ has positive lower density at least
\begin{equation}
\displaystyle C_1\delta e^{-C_2(\log(1/\delta)^{2/3}(\log\log(1/\delta))^{1/3})} \label{boundofCH}
\end{equation}
in the natural numbers. Actually, Ramar\'{e} and Ruzsa \cite{R-R} studied this problem before. They gained general results for subsets of `sifted sequence'. An explicit theorem can be found in \cite{Rama}, showing that the bound in (\ref{boundofCH}) can be replaced by
\begin{displaymath}
C_3\delta/\log\log(1/\delta)
\end{displaymath}
with $C_3$ an absolute constant. Recently, Matom\"{a}ki \cite{Mato} obtained an explicit value of the constant $C_3$ in above estimate.

Now we turn to sparser subsets of the primes. Chipeniuk and Hamel remarked that it is possible to obtain a bound of $\delta^2$ using simple argument involving the Cauchy-Schwarz inequality and Brun's sieve. In this paper, we step a little further.  Let $\mathcal{P}$ be the set of all the primes.

\begin{theorem} \label{TH}
Let $c_0>0$ be arbitrarily real number. Suppose $x$ is sufficiently large. Then for any prime subset $A\subseteq \mathcal{P}\cap [1,x]$ satisfying $|A|= \delta x (\log x)^{-1}$ with $\delta \geq (\log x)^{-c_0}$, we have
\begin{displaymath}
|A+A|\gg_{c_0} \delta x(\log\log x)^{-1}.
\end{displaymath}
\end{theorem}

This result holds uniformly for subsets $A$ with $|A| \geq x(\log x)^{-c_0-1}$. Note that if $A_0=\{p\in \mathcal{P}: p\leq x,\, p\equiv a(\text{mod }q)\}$ with $(a,q)=1$ and  $q\leq (\log x)^c$, then $\delta \sim 1/\varphi(q)$ by the Siegel-Walfisz theorem, where $\varphi$ is Euler's totient function. But $|A_0+A_0|\leq \left|\{n\leq 2x: n\equiv 2a(\text{mod }q)\}\right| \ll x/q$. So roughly speaking we have $|A_0+A_0|\ll \delta x (\log\log q)^{-1}$. The $(\log\log x)^{-1}$-term in Theorem \ref{TH} can not be eliminated and the result is not far from the best possible (maybe it can be replaced by $(\log\log\log x)^{-1}$).

The circle method is applied here. Since too much will be lost if one substitute natural numbers for the primes directly, we make use of the sifted numbers, the integers free of small prime factors, by observing that the exponential sum over the primes shares similar type with that over sifted numbers on major arcs. This trick may be potentially useful for problems that are ``log-sensitive''.

\bigskip

\section{Preliminary Lemmas}

\medskip

Throughout, the letter $p$ always denote a prime and $x$ is a sufficiently large number. We write $f\ll g$ or $f=\textit{O}(g)$ to denote the estimate $|f|\leq c g$ for some positive constant $c$. And we write $f(x)=\textit{o} (g(x))$ for $\lim_{x\rightarrow \infty} f(x)/g(x)=0$. For a set S, we denote by $|S|$ its cardinality. The characteristic function $1_S(x)$ takes value $1$ for $x\in S$ and $0$ otherwise. Write $e(x)=e^{2\pi i x}$. The smallest (or largest) prime factor of an integer $n$ is denoted by $p(n)$ (or $P(n)$, respectively). Define
\begin{displaymath}
\mathcal{B}(x,y):=\{n\leq x:\,p(n)> y\}.
\end{displaymath}
The integers of this set are usually called sifted numbers or rough numbers. For $A\subseteq \mathcal{P}$ and $B\subseteq \mathbb{Z}$, we define the following exponential sums
\begin{displaymath}
S(\alpha)=S(\alpha;x):=\sum\limits_{p\leq x}  e(\alpha p)\log p, \quad S_A(\alpha)=S_A(\alpha;x):=\sum\limits_{p\leq x} 1_A(p)e(\alpha p)\log p,
\end{displaymath}

\begin{displaymath}
T(\alpha)=T(\alpha;x):=\sum\limits_{n\leq x} e(\alpha n), \quad T_B(\alpha)=T_B(\alpha;x):=\sum\limits_{n\leq x} 1_B(n)e(\alpha n).
\end{displaymath}

\medskip

In this section, we present some preliminary lemmas. Lemma \ref{L1} is known as the energy inequality. Lemma \ref{L2} and \ref{L3} are estimates for the minor arcs, while Lemma \ref{L4} and \ref{L5} are used on the major arcs.

\begin{lemma} \label{L1}
Let $A$ be the set defined in Theorem \ref{TH}. Then
\begin{displaymath}
|A+A|\gg \delta^4 x^4I^{-1},
\end{displaymath}
where $I=\int_0^1 |S_A(\alpha)|^4 d\alpha.$
\end{lemma}

\begin{proof}
Recall that $|A|\geq \delta x (\log x)^{-1}$. Write $\{p_n\}$ the sequence of primes. Since $p_n>n(\log n+\log\log n -1)$ for $n\geq 2$ (see \cite{Dus}), it can be deduced that
\begin{displaymath}
\max_{p\in A}\{p\} \geq p_{|A|}\geq \delta x/2.
\end{displaymath}
Combining the prime number theorem, yields
\begin{displaymath}
\sum\limits_{p\in A} \log p \geq \sum\limits_{p\leq p_{|A|}} \log p\geq \sum\limits_{p\leq \delta x/2} \log p \gg \delta x.
\end{displaymath}
By the Cauchy-Schwarz inequality,
\begin{displaymath}
\begin{aligned}
\left|\sum\limits_{p\in A}\log p\right|^4
=&\left|\sum\limits_{n\in A+A}\sum\limits_{p_1,p_2\in A \atop p_1+p_2=n}\log p_1\log p_2\right|^2\\
\leq &|A+A|\cdot  \sum\limits_{n\in A+A}\left|\sum\limits_{p_1,p_2\in A \atop p_1+p_2=n}\log p_1\log p_2\right|^2\\
=&|A+A| \cdot \left(\sum\limits_{p_1,p_2,p_3,p_4\in A \atop p_1+p_2=p_3+p_4}\log p_1\log p_2\log p_3\log p_4\right)\\
=&|A+A| \cdot \int_0^1 |S(\alpha)|^4 d\alpha.
\end{aligned}
\end{displaymath}
Then the lemma follows.
\end{proof}

The next lemma is due up to some details to Vinogradov and stronger versions are now known. However, this one is enough for our proof.

\begin{lemma} \label{L2}
(\cite{Pan-Pan}, Chapter 19, \S1, Corollary 9) Suppose $\displaystyle \alpha=\frac{a}{q}+\beta$, where $(a,q)=1$ and $\displaystyle |\beta|\leq \frac{1}{q^2}$. Then
\begin{displaymath}
S(\alpha)\ll x\log^3 x \left(x^{-1/2}q^{1/2}+q^{-1/2}+\exp\left(-\frac{1}{2}\sqrt{\log x}\right)\right).
\end{displaymath}
\end{lemma}

The following lemma is actually Lemma 4.10 of \cite{G}. Since different kind of notations are used in \cite{G} and we have made a slight change here, the proof is repreduced below.

\begin{lemma} \label{L3}
Suppose $\displaystyle \alpha=\frac{a}{q}+\beta$, where $(a,q)=1$ and $\displaystyle |\beta|\leq \frac{1}{q^2}$. If $y\leq (\log x)^D$ for some absolute constant $D>0$, then
\begin{displaymath}
T_{\mathcal{B}(x,y)}(\alpha,x)\ll q\log q + q^{-1}x\log x +x^{1/2}\log q+x^{1-1/(4D)}.
\end{displaymath}
\end{lemma}

\begin{proof}
  Let $p_1,p_2,\ldots,p_k$ be the primes less than or equal to $y$. By the prime number theory, the magnitude of $k$ can be bounded by $k\ll y/\log y$. The inclusion-exclusion principle yields
\begin{displaymath}
\begin{aligned}
T_{\mathcal{B}(x,y)}(\alpha,x)=&\sum\limits_{n\in \mathcal{B}(x,y)} e(\alpha n)\\
=&\sum\limits_{s=0}^k(-1)^s \sum\limits_{1\leq i_1<\ldots<i_s\leq k} \; \sum\limits_{n\leq x/p_{i_1}\ldots p_{i_s}} e(\alpha p_{i_1}\ldots p_{i_s} n).\\
\ll & \sum\limits_{s=0}^k \sum\limits_{1\leq i_1<\ldots<i_s\leq k} \min\{ x/p_{i_1}\ldots p_{i_s},\;\|\alpha p_{i_1}\ldots p_{i_s}\|^{-1}\}.
\end{aligned}
\end{displaymath}
Let $t$ be a parameter to be specified later and split the sum over $s$ into two parts. For $1\leq s\leq t$, since the product of any $s\leq t$ of $p_1,\ldots,p_k$ is less than or equal to $y^t$ and all such products are distinct, we have
\begin{displaymath}
\begin{aligned}
&\sum\limits_{s=0}^t \sum\limits_{1\leq i_1<\ldots<i_s\leq k} \min\{ x/p_{i_1}\ldots p_{i_s},\;\|\alpha p_{i_1}\ldots p_{i_s}\|^{-1}\}\\
\ll & \sum\limits_{m\leq \, y^t} \min\{ x/m,\;\|\alpha m\|^{-1}\}\\
\ll & q\log q+ xq^{-1} t \log y + y^t \log q.
\end{aligned}
\end{displaymath}
See \cite{Pan-Pan}, chapter 19, \S1, Lemma 2 and Lemma 6 for details of such estimates. For $t+1\leq s\leq k$,
\begin{eqnarray}
&&\;\;\sum\limits_{s=t+1}^k \; \sum\limits_{1\leq i_1<\ldots<i_s\leq k} \min\{ x/p_{i_1}\ldots p_{i_s},\;\|\alpha p_{i_1}\ldots p_{i_s}\|^{-1}\} \nonumber\\
&&\leq  x \sum\limits_{s=t+1}^k\, \sum\limits_{1\leq i_1<\ldots<i_s\leq k} \,  \frac{1}{p_{i_1}\ldots p_{i_s}} \nonumber\\
&& \leq  x \sum\limits_{s=t+1}^k (s!)^{-1}\left(\sum\limits_{i=1}^k p_i^{-1}\right)^s. \label{eqeqeq}
\end{eqnarray}
By one result of Mertens one has $\sum\limits_{i=1}^k p_i^{-1} \leq \log\log y +\textit{O}(1)$. So if $t\geq 3\log \log y$ then \begin{displaymath}
((s+1)!)^{-1}\left(\sum\limits_{i=1}^k p_i^{-1}\right)^{s+1}\leq \frac{1}{2}\cdot (s!)^{-1}\left(\sum\limits_{i=1}^k p_i^{-1}\right)^s.
\end{displaymath}
for $s\geq t+1$. It can be deduced that (\ref{eqeqeq}) is
\begin{displaymath}
\ll x \frac{(2\log\log y)^t}{t!}\ll xt^{-1/2}\left(\frac{2e\log\log y}{t+1}\right)^{t+1} \ll x e^{-t\log t/2},
\end{displaymath}
if we set $ t=[\log x /2\log y]$ (here $[x]$ denotes the integer part of $x$). Then the lemma follows.
\end{proof}

\begin{lemma} \label{L4}
(See \cite{Pan-Pan}, Chapter 20, \S 2) Let $D>0$. Suppose $\displaystyle \alpha=\frac{a}{q}+\beta$, where $(a,q)=1$, $q\leq (\log x)^D$ and $\displaystyle |\beta|\leq \frac{1}{q^2}$. Then
\begin{displaymath}
S(\alpha)=\frac{\mu(q)}{\varphi(q)}\int_2^x e(\beta z) dz +\textit{O}\left(xe^{-c\sqrt{\log x}}(1+|\beta|x)\right).
\end{displaymath}
Here $c$ is a constant only depending on $D$.
\end{lemma}

\begin{lemma} \label{L5}
Suppose $\displaystyle \alpha=\frac{a}{q}+\beta$, where $(a,q)=1$ and $\displaystyle |\beta|\leq \frac{1}{q^2}$. For any constant $D_1>2$, we have
\begin{displaymath}
T_{\mathcal{B}(x,y)}=\frac{\mu(q)}{\varphi(q)} \prod\limits_{p\leq y}\left(1-\frac{1}{p}\right)\int_2^x e(\beta z)dz+\textit{O}\left(xe^{-\frac{1}{3}
\sqrt{\log x}}(1+|\beta|x)\right),
\end{displaymath}
provided that $(\log x)^2 \leq y\leq (\log x)^{D_1}$ and $q< y$.
\end{lemma}

\begin{proof}
Let $u:=\log x/\log y$, $\Phi(x,y):=|\mathcal{B}(x,y)|$ and $\Phi(x,y;a,q):=|\mathcal{B}(x,y;a,q)|$, where
\begin{displaymath}
\mathcal{B}(x,y;a,q):=\{n\in \mathcal{B}(x,y): n\equiv a(\text{mod }q)\}.
\end{displaymath}
It is proved by de Bruijn \cite{Br} (1.13) that the estimate
\begin{displaymath}
\Phi(x,y)=x\prod\limits_{p\leq y}\left(1-\frac{1}{p}\right)\left(1+\textit{O}\left(\log ^3 y \cdot e^{-u(\log u+\log\log u)+c_1 u}\right)\right)
\end{displaymath}
holds uniformly in the range $1\leq u\leq 4y^{1/2}/\log y, \; y\geq 2$, with $c_1$ a constant. And Xuan \cite{Xuan} (Corollary 1) showed that if $D_2>0$ is fixed and $(a,q)=1$, then
\begin{displaymath}
\Phi(x,y;a,q)=\frac{1}{\varphi(q)}\Phi(x,y)\left(1+\textit{O}\left(e^{-\frac{1}{2}
\sqrt{\log x}}\right)\right)
\end{displaymath}
holds uniformly in the range $3/2\leq y\leq x/q$, and
\begin{displaymath}
1<q\leq (\log x)^{D_2},\quad P(q)<y.
\end{displaymath}
Combining the above two estimates gives
\begin{equation}
\Phi(x,y;a,q)=\frac{x}{\varphi(q)}\prod\limits_{p\leq y}\left(1-\frac{1}{p}\right)+\textit{O}\left(x e^{-\frac{1}{2}
\sqrt{\log x}}\right) \label{es_Phiaq}
\end{equation}
for $(\log x)^2\leq y\leq (\log x)^{D_1}$.

For $q<y$, we have
\begin{displaymath}
\begin{aligned}
T_{\mathcal{B}(x,y)}(\alpha)=\sum\limits_{n\in \mathcal{B}(x,y)}e(\alpha n)={\sum\limits_{c=1\atop (c,q)=1}^q} e(ca/q)\sum\limits_{n\in \mathcal{B}(x,y;c,q)}e(n\beta).
\end{aligned}
\end{displaymath}

By partial summation, together with (\ref{es_Phiaq}), we can conclude that
\begin{displaymath}
\sum\limits_{n\in \mathcal{B}(x,y;c,q)}e(n\beta)=\frac{1}{\varphi(q)}\prod\limits_{p\leq y}\left(1-\frac{1}{p}\right) \int_{y}^{x} e(\beta z)dz + \textit{O}\left(x e^{-\frac{1}{2}
\sqrt{\log x}}(1+|\beta|x)\right).
\end{displaymath}
So
\begin{displaymath}
T_{\mathcal{B}(x,y)}=\frac{\mu(q)}{\varphi(q)} \prod\limits_{p\leq y}\left(1-\frac{1}{p}\right)\int_2^x e(\beta z)dz+\textit{O}\left(x e^{-\frac{1}{3}
\sqrt{\log x}}(1+|\beta|x)\right)
\end{displaymath}
and the lemma follows.
\end{proof}

\bigskip

\section{Proof of the theorem}

\medskip

Let $\Delta>0$ be a parameter to be specified later. Put
\begin{displaymath}
y=P=(\log x)^{\Delta},\quad Q=x(\log x)^{-\Delta}.
\end{displaymath}
By Dirichlet's approximation theorem, each $\alpha\in [0,1]$ can be written as
\begin{equation}
\alpha=\frac{a}{q}+\beta,\quad (a,q)=1,\quad 1\leq q\leq Q, \quad |\beta|\leq \frac{1}{qQ}. \label{Dapproximate}
\end{equation} For $a$ and $q$, let $\mathfrak{M}(a,q)$ be the set of $\alpha$ satisfying (\ref{Dapproximate}). Denote the major arcs $\mathfrak{M}$ and the minor arcs $\mathfrak{m}$ by

\begin{displaymath}
\mathfrak{M}=\bigcup\limits_{q< P}\bigcup\limits_{(a,q)=1}\mathfrak{M}(a,q),\quad \mathfrak{m}=\left[0,1\right]\setminus \mathfrak{M}.
\end{displaymath}
The major arcs $\mathfrak{M}(a,q)$ are mutually disjoint whenever $2P\leq Q$.

Note that $\int_{0}^{1} |S_A(\alpha)|^4d\alpha$ represents weighted sum over
\begin{displaymath}
\{(p_1,p_2,p_3,p_4)\in A^4:\; p_1+p_2=p_3+p_4\},
\end{displaymath}
while $\int_{0}^{1} S(\alpha)S_A(\alpha)S_A^2(-\alpha)d\alpha$ does over
\begin{displaymath}
\{(p_1,p_2,p_3,p_4)\in \left(\mathcal{P}\cap [1,x]\right)\times A^3:\; p_1+p_2=p_3+p_4\}.
\end{displaymath}
We can conclude that
\begin{eqnarray}
&&\int_{0}^{1} |S_A(\alpha)|^4 d\alpha \leq \int_{0}^{1} S(\alpha)S_A(\alpha)S_A^2(-\alpha)d\alpha \nonumber \\
&&\quad\quad\quad=\int_\mathfrak{M}S(\alpha)S_A(\alpha)S_A^2(-\alpha)d\alpha+ \int_\mathfrak{m}S(\alpha)S_A(\alpha)S_A^2(-\alpha)d\alpha.\label{eq1}
\end{eqnarray}
Similarly, we have
\begin{eqnarray}
&&\int_{0}^{1} T(\alpha)S_A(\alpha)S_A^2(-\alpha)d\alpha\geq \int_{0}^{1} T_{\mathcal{B}(x,y)}(\alpha)S_A(\alpha)S_A^2(-\alpha)d\alpha \nonumber\\
&&\quad\quad\quad=\int_\mathfrak{M}T_{\mathcal{B}(x,y)}(\alpha)S_A(\alpha)S_A^2(-\alpha)d\alpha+ \int_\mathfrak{m}T_{\mathcal{B}(x,y)}(\alpha)S_A(\alpha)S_A^2(-\alpha)d\alpha. \label{eq2}
\end{eqnarray}
Note that
\begin{displaymath}
\begin{aligned}
&\int_{0}^{1} T(\alpha)S_A(\alpha)S_A^2(-\alpha)d\alpha
=\sum\limits_{n_1\leq x \atop{ p_2,p_3,p_4 \in A
\atop n_1+p_2=p_3+p_4}} \log p_2\log p_3\log p_4\\
& \quad\quad\quad\leq \sum\limits_{p_2\in A}\log p_2\cdot \sum\limits_{p_3\in A}\log p_3 \cdot \sum\limits_{p_4\in A}\log p_4
\leq (|A|\log x)^3
\ll \delta^3 x^3.
\end{aligned}
\end{displaymath}

Now we relate $\int_{0}^{1} S(\alpha)S_A(\alpha)S_A^2(-\alpha)d\alpha$ to $\int_{0}^{1} T_{\mathcal{B}(x,y)}(\alpha)S_A(\alpha)S_A^2(-\alpha)d\alpha$. For $\alpha \in \mathfrak{m}$,
\begin{displaymath}
\begin{aligned}
&\left|\int_\mathfrak{m}S(\alpha)S_A(\alpha)S_A^2(-\alpha)d\alpha\right|\\
\leq &\sup\limits_{\alpha\in \mathfrak{m}}|S(\alpha)|\int_\mathfrak{m}\left| S_A(\alpha)S_A^2(-\alpha)\right|d\alpha\\
\leq &\sup\limits_{\alpha\in \mathfrak{m}}|S(\alpha)|\int_0^1\left| S_A(\alpha)S_A^2(-\alpha)\right|d\alpha\\
\leq &\sup\limits_{\alpha\in \mathfrak{m}}|S(\alpha)| \left(\int_0^1 |S_A(\alpha)|^2 d\alpha \right)^{1/2}\left(\int_0^1 |S_A(\alpha)|^4 d\alpha \right)^{1/2}.
\end{aligned}
\end{displaymath}

By Lemma \ref{L2}, $\sup\limits_{\alpha\in \mathfrak{m}}|S(\alpha)|\ll x(\log x)^{-\Delta/2+3}$. Moreover,

\begin{displaymath}
\int_0^1 |S_A(\alpha)|^2 d\alpha \ll \delta x \log x,
\end{displaymath}
\begin{displaymath}
\int_0^1 |S_A(\alpha)|^4 d\alpha \leq \log x\cdot \int_0^1 T(\alpha)S_A(\alpha)S_A(-\alpha)^2d\alpha \ll \delta^3 x^3\log x.
\end{displaymath}
Then it follows that
\begin{equation}
\left|\int_\mathfrak{m}S(\alpha)S_A(\alpha)S_A^2(-\alpha)d\alpha\right| \ll \delta^2 x^3 (\log x)^{-\Delta/2+4}. \label{eq3}
\end{equation}
Combining Lemma \ref{L3}, which asserts $\sup\limits_{\alpha\in \mathfrak{m}}|T_{\mathcal{B}(x,y)}(\alpha)|\ll x(\log x)^{-\Delta+1}$, similar arguments lead to
\begin{equation}
\left|\int_\mathfrak{m}T_{\mathcal{B}(x,y)}(\alpha)S_A(\alpha)S_A^2(-\alpha)d\alpha\right| \ll \delta^2 x^3 (\log x)^{-\Delta+2}.\label{eq4}
\end{equation}

For $\alpha\in \mathfrak{M}$, Lemma \ref{L4} and Lemma \ref{L5} show that
\begin{displaymath}
T_{\mathcal{B}(x,y)}(\alpha)=S(\alpha)\prod\limits_{p\leq y}\left(1-\frac{1}{p}\right)+\textit{O}\left(\frac{x}{(\log x)^{3\Delta}}\right).
\end{displaymath}
And Mertens' theorem gives $\prod_{p\leq t}\left(1-\frac{1}{p}\right)\sim \frac{e^{-\gamma}}{\log t}$. Then we have
\begin{eqnarray}
&&\,\,\int_\mathfrak{M}S(\alpha)S_A(\alpha)S_A^2(-\alpha)d\alpha \nonumber\\
&&=\sum\limits_{q\leq P}\sum\limits_{(a,q)=1}\int_{a/q-1/qQ}^{a/q+1/qQ}S(\alpha)S_A(\alpha)S_A^2(-\alpha)d\alpha\nonumber\\
&&=\prod\limits_{p\leq y}\left(1-\frac{1}{p}\right)^{-1} \sum\limits_{q\leq P}\sum\limits_{(a,q)=1}\int_{a/q-1/qQ}^{a/q+1/qQ}T_{\mathcal{B}(x,y)}(\alpha)S_A(\alpha)S_A^2(-\alpha)d\alpha\nonumber\\
&&\quad\quad\quad +\textit{O}\left(\frac{x}{(\log x)^{3\Delta}}\cdot \log y \cdot \int_0^1\left|S_A(\alpha)S_A^2(-\alpha)\right| d\alpha \right)\nonumber\\
&&= \prod\limits_{p\leq y}\left(1-\frac{1}{p}\right)^{-1}\int_\mathfrak{M}T_{\mathcal{B}(x,y)}(\alpha)S_A(\alpha)S_A^2(-\alpha)d\alpha+\textit{O}\left(\frac{\delta^2 x^3}{(\log x)^{\Delta-2}}\right). \label{eq5}
\end{eqnarray}
Putting (\ref{eq1})-(\ref{eq5}) together, we conclude that
\begin{displaymath}
\begin{aligned}
\int_{0}^{1} |S_A(\alpha)|^4 d\alpha \leq  &\prod\limits_{p\leq y}\left(1-\frac{1}{p}\right)^{-1}\cdot \int_{0}^{1} T(\alpha)S_A(\alpha)S_A^2(-\alpha)d\alpha+\textit{O}\left(\frac{\delta^2 x^3}{(\log x)^{\Delta/2-4}}\right)\\
\ll &\prod\limits_{p\leq y}\left(1-\frac{1}{p}\right)^{-1}\cdot \delta^3x^3+\textit{O}\left( \frac{\delta^2x^3}{(\log x)^{\Delta/2-4}}\right).
\end{aligned}
\end{displaymath}
Note that $\prod_{p\leq y}\left(1-1/p\right)^{-1}\ll \log y \ll_{\Delta} \log\log x$. Taking $\Delta = 2c_0+8$ so that $(\log x)^{-\Delta/2+4}=\textit{o}(\delta\log y )$, we deduce that
\begin{displaymath}
\int_{0}^{1} |S_A(\alpha)|^4 d\alpha \ll_{c_0}  \delta^3 x^3 \log \log x.
\end{displaymath}
By Lemma \ref{L1}, one easily sees that
\begin{displaymath}
|A+A|\gg \delta^4 x^4 I^{-1}\gg_{c_0} \delta x(\log \log x)^{-1}.
\end{displaymath}
Then Theorem \ref{TH} follows.

\medskip

\acknowledgements{\rm We would like to thank Professor O. Ramar\'{e} and K. Matom\"{a}ki for providing a large amount of information and several discussions on this issue. We thank Professor J. Br\"{u}dern and  Professor T. D. Wooley for some helpful conversations. We are grateful to the anonymous referee for pointing out the errors. Last but not least, the third author would like to show great thanks to Ping Xi for his help and encouragement.}

\end{document}